\documentclass{amsart}

\usepackage{amsmath, amssymb, amsthm}
\usepackage[mathcal]{eucal}
\usepackage{enumerate}
\usepackage{amsfonts}
\usepackage{amscd}
\usepackage{mathrsfs}
\usepackage{graphics}
\usepackage{upgreek}
\usepackage{dirtytalk}

\newtheorem{theorem}{Theorem}

\newtheorem{lemma}{Lemma}

\numberwithin{equation}{section}

\def\supp{{\rm supp\,}}

\def\sinc{{\rm sinc}}
\def\ti{\tilde}

\def\ZZ{\ensuremath{\mathbb Z}}

\def\ZR{\ensuremath{\mathbb R}}

\def\Z1{\ensuremath{\mathbf 1}}
\def\eps{\ensuremath{\epsilon}}

\usepackage[dvipsinames]{xcolor}
\definecolor{midnightblue}{HTML}{0059b3}
\definecolor{chromered}{HTML}{f14233}
 \RequirePackage[colorlinks,citecolor=midnightblue,urlcolor=midnightblue,breaklinks=true,
linkbordercolor=midnightblue,linkcolor=midnightblue]{hyperref}

\makeatletter
\def\env@cases{%
  \let\@ifnextchar\new@ifnextchar
  \left\lbrace
  \def\arraystretch{0.9}%
  \array{@{}l@{\quad}l@{}}}
\makeatother

\title[Lacunary Carleson for NLFT]{Weak $L^2$ bound of the lacunary Carleson operator for the non-linear Fourier transform}
\author{Gevorg Mnatsakanyan}
\curraddr{}
\address{Institute of Mathematics at the National Academy of Sciences of Armenia, Marshal Baghramyan Ave. 24/5, Yerevan 0019, Armenia}
\email{mnatsakanyan\_g@yahoo.com}
\date{July 2025}
\subjclass[2020]{34L40, 42C05}
\keywords{Non-linear Fourier transform, scattering transform, Carleson operators, Krein systems}

\begin{document}

\begin{abstract}
We prove the weak $L^2$ boundedness of a lacunary maximal function of the $SU(1,1)$-valued nonlinear Fourier transform if the potential is in $L^1$.
\end{abstract}

\maketitle

\section{Introduction}

One can write the exponential of the classical Fourier transform $\widehat{f}$ of an integrable function $f$
in terms of a solution to a differential equation. Namely, given $f$ consider the equation
\begin{equation}\label{ODEFT}
    \partial_t G (t, x) = e^{-2ixt}f(t) G(t,x).
\end{equation}
Given initial datum at any point,
a unique solution $G$ exists. With the initial condition $G(-\infty, x) = 1$, we have
\begin{equation*}
    \exp ( \hat{f} (x) ) = G(\infty, x).
\end{equation*}

One can consider matrix-valued analogs of \eqref{ODEFT} such as 
\begin{equation}\label{NLFODE}
    \partial_tG (t, x) =
    \begin{pmatrix}
        0 & e^{-2ixt}f(t) \\
        e^{2ixt} \overline{f(t)} & 0 \\
    \end{pmatrix}
    G(t,x),
\end{equation}
where $G$ is a $2\times 2$ matrix-valued function. It is not difficult to check that $G(t,x)$ takes values in $SU(1,1)$, that is
\begin{equation}\label{matrixNLFT}
    G(t, x) = \begin{pmatrix}
        \overline {a (t, x) } & \overline {b (t, x) } \\
        b(t,x) & a (t, x) \\
    \end{pmatrix},
\end{equation}
where
\begin{equation}\label{aandb}
|a(t,x)|^2 - |b(t,x)|^2 =1.   
\end{equation}

In analogy to the scalar case above, with the initial condition $G(-\infty , x)=I_2$, where $I_2$ is the identity matrix, we call 
the matrix $G(\cdot ) := G(\infty, \cdot)$ the non-linear Fourier transform (NLFT) of $f$. In linear approximation, we have
\begin{align}\label{expansionfora}   a(x) = 1 + O ( \|f \|_1^2) \, , \quad   b(x) = \hat{f} (x) + O ( \|f \|_1^3) \, .
\end{align}

The NLFT \cite{TaoThiele2012} and its kin have long been studied in analysis under various names such as orthogonal polynomials \cite{simon}, Krein systems \cite{Denisov}, scattering transforms \cite{bealscoifman} and AKNS systems \cite{AKNS}. An $SU(2)$ version of the above model in which the lower-left entry of the matrix in \eqref{NLFODE} gets an extra minus sign was studied in \cite{tsai} and has recently been rediscovered \cite{low2017, QSP_NLFA, 5authorsQSP} and found applications in quantum computing and is called quantum signal processing.

A version of the NLFT relevant in this paper are the so-called Krein-de Branges functions
\begin{equation}\label{krein de branges functions}
    E(t,x) := e^{-itx} (a(t,x) + b(t,x)) \text{ and } \ti E(t,x) := e^{-itx} (a(t,x) - b(t,x)) \, ,
\end{equation}
Note, that the pair $(a,b)$ can be recovered from $(E,\ti E)$ and, similar to \eqref{NLFODE}, one can write a system of differential equations that will define the pair $(E,\ti E)$ directly. These functions are the continuous analogs of orthogonal polynomials on the unit circle and possess nice complex analytics properties that we will describe in Section \ref{overview}.

There is abundant literature both from the point of view of the non-linear Fourier transform and of the Krein-de Branges functions. We will only mention several results that are more relevant for us.

There is a non-linear analog of the Plancherel identity,
\begin{equation}\label{nlPlancherel}
    \| \sqrt{ \log |a| } \|_{L^2 (\ZR)} = \sqrt{\frac{\pi}{2}} \| f \|_{L^2 (\ZR)}.
\end{equation}
This formula is proven by a contour integral and in the discrete case goes back to Verblunsky in 1936 \cite{Verblunsky}.

An interesting open problem in the field remains the non-linear Carleson conjecture. That is, for $f\in L^2(\ZR_+)$,
\begin{equation}\label{conjecture}
    | \{ x\in \ZR :\: \sup_t \sqrt{ \log | a (t, x) | } > \lambda \} | \lesssim \frac{1}{\lambda^2} \| f \|_2^2 \, ,
\end{equation}
or its stronger version \cite{denisovSurvery},
\begin{equation}\label{stronger conjecture}
    | \{ x\in \ZR :\: \sup_t |e^{itx}E(t,x)-1| > \lambda \} | \lesssim \frac{1}{\lambda^2} \| f \|_2^2 \, ,
\end{equation}
Like in the linear setting, the bound \eqref{conjecture} implies almost everywhere convergence of $|a(t,\cdot)|$ as $t\to +\infty$. On the other hand, the bound \eqref{stronger conjecture} also implies the convergence $\arg a(t,\cdot)$, which to the best of my knowledge is not implied by \eqref{conjecture}. The inequality \eqref{conjecture} for the Cantor group model of the NLFT was obtained in \cite{MuscaluTaoThiele} by Muscalu, Tao and Thiele. In \cite{MuscaluCounterexample}, the same authors showed that the approach of Christ-Kiselev \cite{CHRIStkiselevmaximal,ChristKiselev2001,christkiselev2002}, by which a Hausdorff-Young and Menshov-Paley-Zygmund type results can be obtained for the NLFT, fails to work in the $L^2$ setting.
There has been a recent attempt by Poltoratski \cite{Poltoratski} to prove almost everywhere convergence of $|a|$. However, an error has been detected by the author and is mentioned in the previous version of this arxiv posting.

In this paper, we prove the weak-$L^2$ boundedness of the lacunary maximal function of the NLFT.
\begin{theorem}\label{maintheorem}
    Let $\|f\|_1\leq 10^{-10}$. Then, we have
    \begin{equation}\label{mainineq}
        |\{s\in \ZR \, : \, \sup_{n} ||E(2^n,s)|-1| > \lambda  \}| \lesssim \frac{1}{\lambda^2 } \|f\|_2^2 \, .
    \end{equation}
\end{theorem}
The restriction on the finiteness of the $L^1$ norm in the above theorem makes it unfit to deduce the almost everywhere convergence of $|a(2^n,\cdot)|$. However, it is strong enough to imply its linear analog. That is, the weak-$L^2$ estimate for the lacunary linear Carleson operator,
\begin{equation}\label{lacunary linear Carleson}
    |\{x\, : \, \sup_n |\mathcal{F}(f\Z1_{[0,2^n]})|>\lambda  \}| \lesssim \frac{1}{\lambda^2} \|f\|_2^2 \, .    
\end{equation}
We show this in Section \ref{appendix}. As the linear Fourier transform has independent symmetries under scaling of the argument and of the value of the function $f$, having inequality \eqref{lacunary linear Carleson} for all $f\in L^2\cap L^1$ automatically implies it for all $f\in L^2$. However, the NLFT has only a one parameter scaling symmetry preserving the $L^1$ norm of the potential \cite[Section 7]{Denisov}, hence, the same implication is not possible and suggests that the restriction on the $L^1$ may be natural. Even then, one would at least hope to replace the constant $10^{-10}$ by an arbitrary constant and let  the implicit constant in \eqref{mainineq} depend on it. If in \eqref{mainineq}, instead of $|E(t,x)|$ one would have $e^{itx}E(t,x)$ like in \eqref{stronger conjecture}, then one could extend the inequality for potentials with small $L^1$ norm to arbitrary $L^1$ norm without much difficulty. However, for \eqref{mainineq} I do not know how to accomplish that with our current technique.

Theorem \ref{maintheorem} is inspired by the recent paper \cite{AMTpolynomials} and by \cite{Poltoratski}. In \cite{AMTpolynomials}, the convergence of an $SU(2)$-valued NLFT along lacunary subsequences is proved. Similar arguments are possible in the setting of $SU(1,1)$ to prove almost everywhere convergence along lacunary subsequences for $f\in L^2(\ZR_+)$.

Lastly let us mention some related work. The paper \cite{bessonovdenisov2021zero} was the first to connect the almost everywhere convergence of the NLFT with the behavior of the zeros of the function $E$. NLFT with sparse lacunary potentials were considered in \cite{rupcic2019,golinski2004}. For a survey of the non-linear analogs of classical inequalities for the Fourier transform, we refer to \cite{diogoSurvey}. For a discussion of various formulations of the non-linear Carleson conjecture we refer to \cite{denisovSurvery}.

The paper is organized as follows. In section \ref{overview} we state the main constructions and some basic lemmas about the function $E$. In section \ref{sec:fejer} we prove an estimate for the reproducing kernel related to $E$. In section \ref{sec:local approximations}, we prove the main sequence of lemmas. They start with approximation formulas using the result of section \ref{sec:fejer} and go up to estimates of the relevant maximal function. Section \ref{sec:main proof} closes the proof of Theorem \ref{maintheorem}.

\textbf{Notation}
We write $A \lesssim B$ if $A\leq CB$ for some absolute constant $C$ and $A\sim B$ if $A\lesssim B$ and $B\lesssim A$.

\section{The function E}\label{overview}
In this section we sum up some of the basic properties of the Krein-de Branges functions \eqref{krein de branges functions}. We refer to \cite{romanov, Remling, Denisov} for an in-depth discussion of these objects and their properties.

Let $\delta_0 = 10^{-10}$. Let us fix the potential $f\in L^2(\ZR_+)\cap L^1(\ZR_+)$ with $\|f\|_1\leq 10^{-10}$. For an entire function $H$, we put $H^\# (z)=\overline{H(\bar z)}$.

$E(t,\cdot )$ is an entire function of exponential type $t$. Furthermore, $E(t,\cdot)$ is in the Paley-Wiener space
$$PW_t := \{ g \in L^2(\ZR) :\: \exists h \in L^2 (-t,t) \text{ with } f(x) = \int_{-t}^t h(\xi) e^{i\xi x} d\xi  \} \, .$$
Also $E(t,\cdot)$ is a Hermite-Biehler function, that is $|E(t,z)|>|E(t,\bar z)|$ for $z\in \mathbb{C}_+$. In particular, all the zeros of $E$ are in the lower half-plane.

The identity \eqref{aandb} is equivalent to
\begin{equation}\label{determinantidentity}
    E\ti E^\# +\ti E E^\# = 2 \, .
\end{equation}
There is an ODE for $E$ that can be easily obtained from \eqref{NLFODE}. We have
\begin{equation}\label{ODEforE}
    \partial_t E (t,z) = -i z E(t,z ) + \overline{f(t)} E^\# (t,z).
\end{equation}
For the scattering function $\mathcal{E} (t,z) := e^{itz} E(t,z) $ we have
\begin{equation}\label{51ofpoltoratskiintro}
    \frac{\partial}{\partial t} \mathcal{E} (t,z) = \overline{f(t)} e^{2izt} \mathcal{E}^{\#} (t,z).
\end{equation}
These differential equations lead to Gr\"onwall's inequalities.
\begin{lemma}\label{gronwalllemma}
\begin{equation}\label{generalizedgronwall}
    |E (t,z) | \leq e^{t|\Im z| + \int_0^t |f(\xi)| d\xi },
\end{equation}
and
\begin{equation}\label{gronwallforscatteringfunction}
    |\mathcal{E} (t_1,z) - \mathcal{E} (t_2,z)| \leq |E(t_1,\bar z)| e^{(t_2-t_1)(|\Im z|-\Im z) + \int_{t_1}^{t_2}|f|} \int_{t_1}^{t_2} |f|.  
\end{equation}
\end{lemma}
The proof of Lemma \ref{gronwalllemma} is presented in Section \ref{appendix}. Let
\begin{equation}
    w(x) = \frac{1}{|a(x)+b(x)|^2} \, ,\text{ and } \ti w(x)=\frac{1}{|a(x)-b(x)|^2} \, .
\end{equation}
The following lemma is a consequence of Lemma \ref{gronwalllemma}.
\begin{lemma}
    We have, for all $x\in \ZR$,
    \begin{equation}\label{scattering function uniform}
        |\mathcal{E} (t,x) - 1 | \leq 2\delta_0 \, ,
    \end{equation}
    \begin{equation}\label{E limit w}
    \frac{1}{|E(t,x)|^2} \to w(x) \, , \text{ as }t\to \infty \, ,
    \end{equation}
    \begin{equation}\label{uniform estimate for w - 1}
        \|w-1\|_\infty \leq 5\delta_0 \, ,
    \end{equation}
    and
    \begin{equation}\label{plancherel with w - 1}
        \|w-1\|_2\lesssim \|f\|_2 \, .
    \end{equation}
\end{lemma}
\begin{proof}
Applying \eqref{gronwallforscatteringfunction} with $t_1=0$ and $t_2=t$, and recalling that $\|f\|_1\leq \delta_0$, we get
$$|1-\mathcal{E}(t,x)| \leq e^{\delta_0}\delta_0 \leq 2\delta_0 \, .$$
Again from \eqref{gronwallforscatteringfunction},
$$
||E(t_1,x)|- |E(t_2,x)|| \leq |E(t_1,x)| \int_{t_1}^{t_2} |f| \, .
$$
Hence, $|E(t,x)|$ is Cauchy, and converges to $|a(x)+b(x)|=1/\sqrt{w(x)}$.
From \eqref{scattering function uniform},
$$\left| 1-\frac{1}{|E|^2} \right| = \frac{|1-|E||(1+|E|)}{|E|^2} \leq \frac{2\delta_0 (2+2\delta_0)}{(1-2\delta_0)^2}\leq 5\delta_0 \, .$$
Passing to a limit by \eqref{E limit w} proves \eqref{uniform estimate for w - 1}.

Let us prove \eqref{plancherel with w - 1}. By \eqref{krein de branges functions} and \eqref{determinantidentity}, we have
$$
|a(t,x)|^2 = \frac{1}{4}\left( |E(t,x)|^2 + |\ti E(t,x)|^2 + 2\Re (E(t,x)\overline{E(t,x)}) \right) 
$$
$$= \frac{1}{2}\left( |E(t,x)|^2 + |\ti E(t,x)|^2 + 2 \right) \, .
$$
Again passing to a limit, we have
$$
|a(x)|^2 = \frac{1}{4} \left( \frac{1}{w(x)} + \frac{1}{\ti w(x)} + 2 \right) \, .
$$
By \eqref{determinantidentity},
$$
|E(x)\ti E(x)| \geq |\Re E(x)\overline{\ti E(x)} | = 1 \, .
$$
Passing to a limit, we have $w(x)\ti w(x)\leq 1$, hence,
$$
|a(x)| \geq \sqrt{ \frac{1}{4} (w+\frac{1}{w}+2)} =  \frac{\sqrt{w} + 1/\sqrt{w}}{2} = 1 + \frac{(\sqrt{w}-1)^2}{2\sqrt{w}}
$$
$$
= 1 + \frac{(w-1)^2}{2\sqrt{w}(\sqrt{w}+1)^2} \geq 1 + \frac{(w-1)^2}{10} \, .
$$
The last estimate with \eqref{nlPlancherel} and \eqref{uniform estimate for w - 1} implies \eqref{plancherel with w - 1}.
\end{proof}

The function $w$ is the absolutely continuous part of the spectral measure of the system \eqref{NLFODE}. As $f\in L^1(\ZR_+)$, the spectral measure is absolutely continuous. On the other hand, the continuous analog of Szeg\"o's theorem states that if $f\in L^2(\ZR_+)$, then $\log w$ is Poisson finite.

Let us introduce the scalar product weighted by $w$ in the natural way,
$$
\langle F, G\rangle_w = \int_{\ZR} F(x)\overline{G(x)} w(x)dx \, .
$$
Then, the Paley-Wiener spaces $PW_t$ weighted by $w$ become the so-called de Branges spaces which are usually denoted by $B(E)$. The function
\begin{equation}\label{reproducingkerneldef}
    K(t, \lambda, z) := \frac{i}{2\pi} \frac{E(t,z) E^{\#} (t,\lambda) - E^{\#} (t,z) E(t, \lambda )}{z-\lambda}
\end{equation}
is a reproducing kernel for $B(E)$. That is, $K(t,\lambda,\cdot) \in PW_t$ and for any $F\in PW_t$,
\begin{equation}\label{reproducingkernelprop}
    F(\lambda) = \langle F, K(t, \lambda, \cdot) \rangle_w \, .
\end{equation}

When $f\equiv 0$, then $E(t,z) = e^{-itz}$, $w\equiv 1$ and $B(E)$ just coincides with the Paley-Wiener space $PW_t$, for which the reproducing kernel is the sinc function
\begin{equation*}
    \sinc (t, \lambda, z) := \frac{1}{\pi } \frac{ \sin t (z - \lambda ) }{ z - \lambda}.
\end{equation*}

\section{An estimate for the reproducing kernel}\label{sec:fejer}
Let $Mh$ denote the Hardy-Littlewood maximal function of a locally integrable function $h$. For $s\in \ZR, t>0$ let $I_{s,t} = [s-2\pi/t,s+2\pi/t]$.

\begin{lemma}\label{fejerestimatemaximal}
For all $s\in \ZR$, $t>0$, we have
\begin{equation}\label{eqq216}
    \sup_{x,y \in I_{s,t} }\Big| K(t,y,x) - \frac{1}{w(s)} \sinc (t,y,x) \Big| \lesssim t M(w-1) (s).
\end{equation}
\end{lemma}

By Christoffel-Darboux formula \cite[Lemma 3.6]{Denisov},
\begin{equation}\label{christoffeldarboux}
    K(t,y,x)=2e^{-it(x-y)}\int_0^{2t} e^{i\xi (x-y)} E(\xi,x) \overline{E(\xi, y)} d\xi
\end{equation}
In linear approximation \eqref{expansionfora}, we see that
\begin{equation}
    K(t,s,s) = 2t(1+2\frac{1}{t}\int_0^t \Re \mathcal{F}(f\Z1_{[0,\xi]})(s) d\xi + O(\|f\|_1^2) ),
\end{equation}
The main term in the above display,
$$\frac{1}{t}\int_0^t \Re \mathcal{F}(f\Z1_{[0,\xi]})(s) d\xi \, ,$$
is a Fejer mean of the linear Fourier transform of $f$. So Lemma \ref{fejerestimatemaximal} can be understood as a non-linear version of the estimate for the Fejer mean.

Our proof relies on several applications of Cauchy-Schwarz inequalities and on \eqref{uniform estimate for w - 1}. However, the result is less trivial if we drop the assumption $f\in L^1$. The corresponding qualitative convergence result for the orthogonal polynomials goes back to \cite{matenevaitotik}. Its continuous analog is proved in \cite{Gubkin}. See also \cite{Bessonov} for related results.
\begin{proof}
Let us first prove the diagonal case. Assume $x=y \in I_{s,t}$ is fixed. Put
\begin{equation*}
    s_t (u) := \frac{ |\sinc (t, y, u) |^2}{\|\sinc (t, y, \cdot) \|_2^2}.
\end{equation*}
It is easy to see
$$s_t (u) \lesssim \frac{t}{t^2 |u-y|^2 + 1 } \text{ and } \int_\ZR s_t(u)^2 du = 1,
$$
hence
$$\int_\ZR s_t(u)|(w-1)(u)|du \leq M(w-1)(s).$$
We use the reproducing kernel property \eqref{reproducingkernelprop} and a  Cauchy-Schwarz to write
$$
\sinc(t,y,y)^2=|\int_\ZR \sinc (t,y,u) \overline{K(t,y,u)} w(u)du |^2
$$
$$
\leq \int_\ZR |\sinc (t,y,u)|^2 w(u)du \int_\ZR |K(t,y,u)|^2 w(u)du
$$
$$
=K(t,y,y) \int_\ZR |\sinc (t,y,u)|^2 w(u)du
$$
$$\leq K(t,y,y) \sinc (t,y,y) \int_\ZR s_t(u)w(u)du
$$
$$
\leq K(t,y,y) \sinc (t,y,y) (w(s)+ M(w-1)(s)))\, .$$
Thus,
$$
K(t,y,y) - \frac{1}{w(s)}\sinc (t,y,y) \gtrsim - tM(w-1)(s) \, .
$$
For the upper bound, we use the reproducing kernel property of the $\sinc$ function.
$$
K(t,y,y)^2=|\int_\ZR K(t,y,u) \overline{\sinc (t,y,u)} du |^2
$$
$$
\leq \int_\ZR |K(t,y,u)|^2 w(u)du \int_\ZR |\sinc (t,y,u)|^2 \frac{du}{w(u)}
$$
$$
= K(t,y,y) \sinc (t,y,y) \int_\ZR s_t(u) \frac{du}{w(u)}
$$
$$
= K(t,y,y) \sinc (t,y,y) \left( \int_\ZR s_t(u) \frac{du}{w(u)} -\frac{1}{w(s)} + \frac{1}{w(s)} \right)
$$
$$\leq K(t,y,y) \sinc (t,y,y) \Big( \frac{1}{w(s)}+ M (w-1) (s) \Big).$$

We move to arbitrary $x,y \in I_{s,t}$. We have
$$
\langle K(t, y, \cdot) - \frac{1}{ w(s) } \sinc(t, y,\cdot) , K(t, y, \cdot) - \frac{1}{ w(s) } \sinc(t, y ,\cdot) \rangle_{w} = $$
$$
= K(t,y,y) + \frac{\| \sinc(t,y,\cdot) \|_2^2}{w(s)^2} \int s_t wdu - \frac{2}{ w(s)} \sinc (t,y,y)
$$
$$
= K(t,y,y) - \frac{1}{w(s)} \sinc (t,y,y) + \frac{\sinc (t,y,y)}{ w(s)^2 } (\int s_t wdu - w(s)) )
$$
$$
\lesssim  \sinc(t,y,y) M(w-1)(s) ,
$$
That is,
\begin{equation}\label{eq:3.1norm}
    \| K(t, \lambda, \cdot) - \frac{1}{w(s)} \sinc (t,y,\cdot ) \|_{w}^2 \lesssim  t M(w-1)(s).
\end{equation}
Then,
$$
K(t,y,x) - \frac{1}{w(s)} \sinc (t,y,x)
=
\langle K(t,y,\cdot) , K(t,x,\cdot) \rangle_w - \frac{1}{w(s)} \sinc (t,y,x)
$$
$$
=
\langle K(t,y,\cdot) - \frac{1}{w(s)} \sinc (t,y,\cdot) , K(t,x,\cdot) - \frac{1}{w(s)} \sinc (t,x,\cdot) \rangle_w + \frac{1}{w(s)} \sinc (t,y,x)
$$
$$+\frac{1}{w(s)} \overline{\sinc (t,x,y)}
-\frac{1}{w(s)^2} \int_{\ZR} \sinc (t,y,u)\overline{\sinc(t,x,u)} w(u) du - \frac{1}{w(s)} \sinc (t,y,x)
$$
$$
= \langle K(t,y,\cdot) - \frac{1}{w(s)} \sinc (t,y,\cdot) , K(t,x,\cdot) - \frac{1}{w(s)} \sinc (t,x,\cdot) \rangle_w
$$
$$
+ \frac{1}{w(s)} \sinc (t,y,x) -\frac{1}{w(s)^2} \int_{\ZR} \sinc (t,y,u)\overline{\sinc(t,x,u)} w(u) du.
$$
For the first term above, we use Cauchy-Schwarz and \eqref{eq:3.1norm}.
$$
\left|\langle K(t,y,\cdot) - \frac{1}{w(s)} \sinc (t,y,\cdot) , K(t,x,\cdot) - \frac{1}{w(s)} \sinc (t,x,\cdot) \rangle_w \right| \leq
$$
$$
\| K(t,y,\cdot) - \frac{1}{w(s)} \sinc (t,y,\cdot)\|_w \| K(t,x,\cdot) - \frac{1}{w(s)} \sinc (t,x,\cdot) \|_w
$$
$$\lesssim tM(w-1)(s).
$$
For the second term, by the reproducing kernel property \eqref{reproducingkernelprop} for $\sinc$ and by Cauchy-Schwarz we write
$$
\left| \frac{1}{w(s)} \sinc (t,y,x) -\frac{1}{w(s)^2} \int_{\ZR} \sinc (t,y,u)\overline{\sinc(t,x,u)} w(u) du \right|=
$$
$$
\left| \frac{1}{w(s)^2} \int_{\ZR} \sinc (t,y,u)\overline{\sinc(t,x,u)} (w(u)-w(s))du \right| \leq t M(w-1)(s).
$$
\end{proof}

\section{Local approximations from Lemma \ref{fejerestimatemaximal}}\label{sec:local approximations}
Let us introduce the slightly unusual notation $A=B+O_C(D)$ for $|A-B|\leq CD$. This will help us to take care of the constants while, hopefully, keeping the intuitive flow of the computations.

Let
$$A_{t,s} = \frac{e^{its}}{2} (E(t,s)+iE(t,s+\pi/2t) \text{ and } B_{t,s} =\frac{e^{-its}}{2} (E(t,s)-iE(t,s+\pi/2t)) \, .$$
By \eqref{scattering function uniform}, we get
\begin{equation}\label{A and B uniform estimates}
    |A_{t,s} - 1 | \leq 2\delta_0 \, , \quad |B_{t,s}| \leq 2\delta_0 \, .
\end{equation}
Similarly, we define $\ti A_{t,s}$ and $\ti B_{t,s}$.

Let us denote by $C_1$ the maximum of the two implicit absolute constants in Lemma \ref{fejerestimatemaximal} for $E$ and $\ti E$. Namely, for any $x,y \in I_{s,t}$,
\begin{equation}\label{Fejer estimate with O sign}
    K(t, y,x) = \frac{1}{w(s)} \sinc (t,y,x) + O_{C_1} (t M(w-1) (s) ) \, ,
\end{equation}
and
\begin{equation}
    \ti K(t,y,x) = \frac{1}{\ti w(s)} \sinc (t,y,x) + O_{C_1} (t M(\ti w-1) (s) ) \, .
\end{equation}

\begin{lemma}\label{1st local estimate}
    For all $s\in \ZR$, $t>0$,
    $$
    \sup_{x\in I_{s,t} } |E(t,x) - (A_{t,s} e^{-itx}  + B_{t,s} e^{itx} ) | \leq 120 C_1 M(w-1) (s) \, ,
    $$
    and
    $$
    \sup_{x\in I_{s,t} } |\ti E(t,x) - (\ti A_{t,s} e^{-itx}  + \ti B_{t,s} e^{itx} ) | \leq 120 C_1 M(\ti w-1) (s) \, .
    $$
\end{lemma}

\begin{proof}
    By \eqref{reproducingkerneldef} and \eqref{Fejer estimate with O sign}, for $x,y \in I_{s,t}$,
    $$
    \frac{i}{2\pi} \frac{E(t,x)E^\#(t,y) - E^\#(t,x)E(t,y)}{x-y} = \frac{1}{\pi} \frac{\sin t(x-y)}{x-y} + O_{C_1}(tM(w-1)) \, .
    $$
    As $|x-y|\leq 4\pi/t$,
    \begin{equation}\label{eq:z and lambda}
        E(t,x)E^\#(t,y) - E^\#(t,x)E(t,y) = e^{-it(x-y)} - e^{it(x-y)}  + O_{8\pi^2C_1}(M(w-1)) \, .    
    \end{equation}
    Consider the two equations \eqref{eq:z and lambda} for the pairs $(x,y)=(x,s)$ and $(x,s+\pi/2t)$ as a linear system in $E(t,x)$ and $E^\#(t,x)$. Solving it, we get
    $$
    E(t,x) \Big(E^\#(t,s)E(t,s+\pi/2t)-E^\#(t,s+\pi/2t)E(t,s) \Big)
    = e^{-it(x-s)}\big( E(t,s+\pi/2t) $$$$-i E(t,s) \big) - e^{it(x-s)} \big( E(t,s+\pi/2t) + iE(t,s) \big) + O_{24\pi^2C_1}(M(w-1)) \, .
    $$
    The expression in the brackets on the first line above is equal to the left hand side of \eqref{eq:z and lambda} for the pair $(x,y)=(s+\pi/2t,s)$. Hence, we have
    $$
    E(t,x) \left( e^{-i\pi/2} - e^{i\pi/2} + O_{8\pi^2C_1}(M(w-1)) \right) = -2i A_{t,s}e^{-itx}
    $$
    $$-2i B_{t,s}e^{itx} + O_{24\pi^2C_1}(M(w-1)) \, .
    $$
    Dividing both sides by $-2i$ and using $|E(t,x)|\leq 1+5\delta_0$ by \eqref{generalizedgronwall}, we conclude the proof of the lemma.
\end{proof}

Let us fix an arbitrary $0< \eps < 1$ and denote
\begin{equation}\label{S set}
    S_\eps := \{ s\in \ZR \, : \, M(w-1)(s) + M(\ti w - 1)(s) < 120^{-1}C_1^{-1} \delta_0 \eps \} \, .
\end{equation}
The following lemma adjusts the parameters $A$ and $B$.
\begin{lemma}\label{adjusting parameters}
    For $s\in S_\eps$ and any $t>0$, we have
    \begin{equation}
        \left| |A_{t,s}|^2 - \frac{1}{w} - |B_{t,s}|^2 \right| \leq 18\delta_0 \eps \, ,
    \end{equation}
    \begin{equation}\label{ti A in terms of A}
        \left| \ti A_{t,s} - \frac{\sqrt{w}}{\sqrt{\ti w}} A_{t,s} e^{\pm i \arccos \sqrt{w\ti w}} \right| \leq 150 \delta_0 \eps \, ,
    \end{equation}
    and
    \begin{equation}
        \left| \ti B_{t,s} + \frac{\sqrt{w}}{\sqrt{\ti w}} B_{t,s} e^{\mp i \arccos\sqrt{w\ti w}} \right| \leq 10\delta_0 \eps \, ,
    \end{equation}
    where in the last two displays for $\pm$ and $\mp$ we either take the first signs in both or the second signs in both.
\end{lemma}
\begin{proof}
    Let us omit the subscripts of $A$'s and $B$'s for simplicity. By the determinant identity \eqref{determinantidentity},
\begin{equation}
    2\Re (A\overline{\ti A}+B\overline{\ti B}) + 2\Re \left( (\ti A \bar B + A \overline{\ti B}) e^{-2itz} \right) = 2 + O_{6\delta_0}(\eps) \, .
\end{equation}
Hence,
\begin{equation}\label{abs value equal 0}
    |\ti A \bar B + A \overline{\ti B}| \leq 6 \delta_0 \eps \, ,
\end{equation}
and
\begin{equation}\label{Re equal 1}
    \Re (A\overline{\ti A}+\bar B\ti B) = 1 + O_{9\delta_0}(\eps) \, . 
\end{equation}
From \eqref{abs value equal 0}, we have,
\begin{equation}\label{ti B in terms of B and As}
    \ti B = - B \frac{\overline{ \ti A}}{\bar A} + O_{7\delta_0}(\eps) \, .
\end{equation}
On the other hand, from Lemma \ref{fejerestimatemaximal}, for $z,\lambda \in (s-2\pi/t,s+2\pi/t)$,
$$
(|A|^2-|B|^2)(e^{-it(\lambda-z)} -e^{it(\lambda-z)}) + \Im (A\bar B) (e^{-it(\lambda+z)} - e^{it(\lambda +z)}) 
$$
$$= \frac{1}{w(s)} (e^{-it(\lambda-z)} -e^{it(\lambda-z)})  + O_{6\delta_0}(\eps) \, .
$$
From the latter we deduce
$$|\Im (A\bar B)| \leq 12 \delta_0 \eps \, ,$$
and
\begin{equation}\label{A2-B2 equal 1/w}
    |A|^2-|B|^2 = \frac{1}{w} + O_{18\delta_0}(\eps) \, .
\end{equation}
Plugging \eqref{ti B in terms of B and As} into \eqref{A2-B2 equal 1/w} for $\ti A,\ti B$ and $\ti w$, we get
$$
|\ti A|^2-|\ti B|^2 = |\ti A|^2 \frac{|A|^2-|B|^2}{|A|^2} + O_{\delta_0}(\eps) = \frac{1}{\ti w} + O_{18\delta_0}(\eps) \, .
$$
Using in \eqref{A2-B2 equal 1/w},
$$
\frac{|\ti A|^2}{|A|^2} \left( \frac{1}{w}+O_{18\delta_0}(\eps) \right) = \frac{1}{\ti w} + O_{19\delta_0}(\eps) \, .
$$
So,
$$
\frac{|\ti A|}{|A|} = \frac{\sqrt{w}}{\sqrt{\ti w}} + O_{40\delta_0}(\eps) \, .
$$
Finally, plugging \eqref{ti B in terms of B and As} and the above relation into \eqref{Re equal 1}, we write
$$
\Re \left( A\overline{\ti A} - |B|^2 \frac{\overline{\ti A}}{\bar A} \right) = 1 + O_{10\delta_0}(\eps)\, .
$$
$$\Re \left( \frac{\ti A}{A} \right) = w + O_{31\delta_0}(\eps) \, .
$$
So,
$$
\frac{\sqrt{w}}{\sqrt{\ti w}} \cos \left( \arg \ti A - \arg A \right) = w +O_{71\delta_0}(\eps) \, .
$$
And we conclude
\begin{equation}
    \arg \ti A - \arg A = \pm \arccos \sqrt{w\ti w} + O_{75\delta_0}(\eps) \, .
\end{equation}
Plugging this back into \eqref{ti B in terms of B and As}, we get
$$
\ti B = -B \frac{\sqrt{w}}{\sqrt{\ti w}} e^{\mp \arccos \sqrt{w\ti w}} + O_{10\delta_0} (\eps) \, .
$$
And the proof of the lemma is complete.
\end{proof}

Let $T:\ZR \to \ZR_+$ be an arbitrary measurable function. Define
\begin{equation}\label{R set}
    R_\eps^T := \{s\in S_{\eps} \, :\, |B_{s,T(s)}|< \eps \} \, .  
\end{equation}

\begin{lemma}\label{E and 1 over root w}
    If $s\in R_\eps^T$, then
    \begin{equation}
        \left| |E(T(s),s)| - \frac{1}{\sqrt{w}} \right| \leq 3\eps \, .
    \end{equation}
\end{lemma}
\begin{proof}
    By Lemma \ref{adjusting parameters}, we have
    $$
    |E(T(s),s)| = O_{\delta_0}(\eps) + \left| A_{T(s),s} + B_{T(s),s}e^{2its} \right|
    $$
    $$
    = O_{\delta_0+1}(\eps) + \sqrt{\frac{1}{w}+|B_{T(s),s}|^2+O_{18\delta_0}(\eps)}
    $$
    $$
    =O_{20\delta_0+2} (\eps) + \frac{1}{\sqrt{w}} \, ,
    $$
    and the lemma is proved.
\end{proof}

To deal with the points $s\in S_\eps \cap (R_\eps^T)^c$, we need the following simple lemma.

\begin{lemma}\label{osc double osc}
    Let $a> \frac{1}{2}$, $0< b<\frac{1}{4}$, $0\leq c < \frac{1}{4}$ and $x_1, x_2\in \ZR$, then there exists a set $U \subset [-\pi,\pi]$ such that $|U| \geq 10^{-4}$ and for all $u\in U$
    \begin{equation}
        \left| |a + b e^{2i(u-x_1)} + c e^{-i(u-x_2)} | - 1 \right| \geq 10^{-3} \max (b, c) \, .
    \end{equation}
\end{lemma}
\begin{proof}
Let
$$f(x) = |a + b e^{2i(x-x_1)} + c e^{-i(x-x_2)}|^2$$
$$
= (a+ b\cos 2(x-x_1) + c\cos (x-x_2) )^2 + (b\sin 2(x-x_1) -c\sin (x-x_2))^2
$$
$$
=a^2 + b^2 + c^2 + 2ab \cos 2(x-x_1) + 2ac\cos (x-x_2) + 2bc \cos \big( 3x-2x_1-x_2 \big) \, .
$$
Then,
$$
f'(x) = -4ab \sin 2(x-x_1) - 2ac \sin (x-x_2) - 6bc \sin (3x-2x_1-x_2) \, .
$$
We compute
$$
\int_{-\pi}^\pi |f'(x)|^2dx =
16a^2b^2 + 4a^2c^2 + 36b^2c^2
\geq b^2+c^2 \, .
$$
Hence, there exists an $x_0$ such that
$$
|f'(x_0)| \geq \sqrt{b^2+c^2} \geq \max (b,c) \, .
$$
As we can estimate
$$
|f''(x)| \leq 8ab + 2ac + 18bc \leq 8(b+c) \leq 16\max (b,c) \, ,
$$
for $x\in X=(x_0 - \frac{1}{50},x_0+\frac{1}{50})$ we have
$$
|f'(x)| \geq |f(x_0)| - |x-x_0|\sup_{\xi \in X} |f''(\xi)| \geq \frac{1}{2}\max (b,c) \, .
$$
The last inequality also implies that $f'(x)$ maintains the sign on $X$ and $b$ is strictly positive. We deduce,
$$
|f(x_0-\frac{1}{50})-f(x_0+\frac{1}{50})| \geq \frac{1}{100} \inf_{\xi \in X} |f'(\xi)| \geq \frac{1}{200} \max (b,c) \, .
$$
Therefore, for either $u_0 = x_0-\frac{1}{50}$ or $u_0=x_0+\frac{1}{50}$ we have
$$|f(u_0)-1| \geq \frac{1}{200} \max (b,c) \, .$$
As $|f'(x)| \leq 4\max (b,c)$ and also $|\sqrt{f(x)}-1| \geq \frac{1}{2}|f(x)-1|$, we conclude that $U=(u_0-10^{-4},u_0+10^{-4})$ satisfies the conclusion of the lemma.
\end{proof}

For $0\leq t_1< t_2$, let $E_{t_1\to t_2} (x)$ denote the function $E(t_2,x)$ corresponding to potential $f\Z1_{(t_1,t_2)}$. Similarly, we will define the NLFT matrix $G_{t_1\to t_2}$. Then,
$$
G(t_2,x) = G_{t_1\to t_2} (x) G(t_1,x) \, .
$$
Thus, recalling also \eqref{krein de branges functions}, we can express $E_{t_1\to t_2}(x)$ in terms of $E(t_2,x)$, $\ti E(t_2,x)$, $E(t_1,x)$ and $\ti E(t_1,x)$. Namely, 
$$
E_{t_1\to t_2}(u) = \frac{1}{2} \Big( E(t_2,u) \ti E^\#(t_1,u) + E(t_2,u) \ti E(t_1,u) + \ti E(t_2,u)E^\# (t_1,u) $$
\begin{equation}\label{E t1 t2 formula}
    - \ti E(t_2,u)E(t_1,u) \Big)\, .  
\end{equation}

\begin{lemma}\label{E bigger than eps}
    If $s\in S_\eps \cap (R_\eps^T)^c$, then there exists $U\subset I_{s,T(s)}$ such that $|U|\geq 10^{-4}/T(s)$, and for all $u\in U$,
    \begin{equation}
        \left| |E_{T(s)/3\to T(s)} (u)| - 1 \right| \geq 10^{-5} \eps \, .
    \end{equation}
\end{lemma}
\begin{proof}
    Let us denote $t_0 := T(s)/3$, $A_1:= A_{t_0,s}$, $A_2:=A_{3t_0,s}$ and so on. Also let $\varphi := \arccos\sqrt{w (s)\ti w(s)}$ and $\eta_1$ and $\eta_2$ be the signs in \eqref{ti A in terms of A} for $t=t_0$ and $t=3t_0$.
    
    Plugging $t_1=t_0, t_2=3t_0$ into \eqref{E t1 t2 formula}, and applying the approximations of Lemma \ref{1st local estimate} and Lemma \ref{adjusting parameters}, we write
    $$
    E_{t_0\to 3t_0}(u) = O_{25\delta_0}(\eps) + \frac{1}{2}\Big( e^{-i2t_0u}(A_2\overline{\ti A_1} + \ti A_2\bar A_1 + A_2\ti B_1 -\ti A_2B_1)
    $$
    $$
    +e^{i2t_0u} (B_2\overline{\ti B_1} +\ti B_2 \bar B_1 + B_2\ti A_1 - \ti B_2 A_1) + e^{-i4t_0u} (A_2\overline{\ti B_1}+ A_2\ti A_1 + \ti A_2\bar B_1 - \ti A_2 A_1)
    $$
    $$
    + e^{4it_0u} (B_2\overline{\ti A_1} + \ti B_2 \bar A_1 + B_2\ti B_1 - \ti B_2B_1 )\Big)
    $$
    $$
    = O_{10^5\delta_0} (\eps) + \frac{1}{2}\frac{\sqrt{w}}{\sqrt{\ti w}} \Big( e^{-2t_0u} A_2(\bar A_1 - B_1) (e^{-i\eta_1\varphi} + e^{i\eta_2\varphi}) 
    $$
    $$
    + e^{2it_0u} B_2(A_1-\bar B_1) (e^{i\eta_1 \varphi} + e^{-i\eta_2\varphi}) + e^{-4it_0u} A_2(A_1-\bar B_1) (e^{i\eta_1\varphi} -e^{i\eta_2\varphi})
    $$
    \begin{equation}\label{eq:4.12}
        + e^{4it_0u} B_2 (\bar A_1 - B_1) (e^{-i\eta_1\varphi} - e^{-i\eta_2\varphi}) \Big) \, .
    \end{equation}
    We want to apply the previous lemma. Let
    $$\theta_0 = \arg \Big( A_2(\bar A_1-B_1)(e^{-i\eta_1 \varphi} + e^{i\eta_2\varphi} ) \Big) \, ,$$
    and we choose $a,b,c, x_1,x_2$ such that
    $$ae^{i\theta_0} = \frac{\sqrt{w}}{2\sqrt{\ti w}} A_2(\bar A_1-B_1)(e^{-i\eta_1 \varphi} + e^{i\eta_2\varphi}) \, ,$$
    $$
    be^{-2ix_1 + i\theta_0} = \frac{\sqrt{w}}{2\sqrt{\ti w}} B_2(A_1-\bar B_1) (e^{i\eta_1\varphi} + e^{-i\eta_2\varphi}) \, ,
    $$
    $$
    ce^{ix_2 + i\theta_0} = \frac{\sqrt{w}}{2\sqrt{\ti w}} A_2(A_1-\bar B_1) (e^{i\eta_1\varphi} - e^{i\eta_2\varphi}) \, .
    $$
    By \eqref{A and B uniform estimates} and \eqref{uniform estimate for w - 1},
    $$
    a\geq \frac{\sqrt{1-5\delta_0}}{2\sqrt{1+5\delta_0}} (1-\delta_0)(1-4\delta_0) (1-5\delta_0) \geq \frac{1}{2} \, . $$
    Also,
    $$
    b \leq \frac{\sqrt{1+5\delta_0}}{\sqrt{1-5\delta_0}} 2\delta_0 (1+4\delta_0) \leq \frac{1}{4} \, ,
    $$
    and
    $$
    c\leq \frac{\sqrt{1+5\delta_0}}{\sqrt{1-5\delta_0}} (1+2\delta_0)(1+4\delta_0)\sqrt{1-(1-5\delta_0)^2} \leq \frac{1}{2} \, .
    $$
    Furthermore, as $s\in (R_\eps^T)^c$, we have $|B_2|>\eps$, so
    $$
    b \geq \eps \frac{\sqrt{1-5\delta_0}}{\sqrt{1+5\delta_0}} (1-4\delta_0) (1-5\delta_0) \geq \frac{\eps}{2} \, ,
    $$
    and
    $$
    \frac{\sqrt{w}}{2\sqrt{\ti w}} |B_2(\bar A_1-B_1) (e^{-i\eta_1 \varphi} - e^{-i\eta_2 \varphi}) | \leq \frac{\sqrt{w}}{2\sqrt{\ti w}} |B_2(\bar A_1-B_1)| |\sin \varphi | 
    $$
    \begin{equation}\label{fourth term estimate by b}
    = \frac{\sqrt{w}}{2\sqrt{\ti w}} |B_2(\bar A_1-B_1)| \sqrt{1-w\ti w} \leq 4\sqrt{\delta_0} \frac{\sqrt{w}}{2\sqrt{\ti w}} |B_2(\bar A_1-B_1)| \leq 10^{-4} b \, ,  
    \end{equation}
    as $\delta_0 = 10^{-10}$.

    By Lemma \ref{osc double osc}, \eqref{eq:4.12} and \eqref{fourth term estimate by b},
    $$
    ||E_{t_0\to 3t_0} (u)| - 1 | \geq 10^{-3} b -10^5 \delta_0\eps  - \frac{\sqrt{w}}{2\sqrt{\ti w}} |B_2(\bar A_1-B_1) (e^{-i\eta_1 \varphi} - e^{-i\eta_2 \varphi}) | 
    $$
    $$
    \geq 10^{-4} b - 10^5\delta_0\eps > 10^{-5}\eps \, ,
    $$
    as $\delta_0 = 10^{-10}$.
\end{proof}

\section{Proof of Theorem \ref{maintheorem}}\label{sec:main proof}

Let $T\, : \, \ZR\to \ZR_+$ be an arbitrary measurable function. For any $1> \eps >0$ let
$$
F_\eps^T := \{s \, : \, ||E(T(s),s)|-1| > \eps \} \, .
$$
We want to prove
$$
|F_\eps^T| \lesssim \frac{1}{\eps^2} \|f\|_2^2 \, ,
$$
with the implicit constant independent of $T$.

Recalling the sets \eqref{S set} and \eqref{R set}, we estimate
$$
|F_\eps^T| \leq |F_\eps^T\cap (S_\eps)^c| + |F_\eps^T\cap R_\eps^T| + \left| F_\eps^T \setminus \big( (S_\eps^c)\cup R_\eps^T \big) \right| \, .
$$
By the $L^2$ estimate for the maximal function and \eqref{plancherel with w - 1},
$$
|S_\eps^c| \lesssim \frac{1}{\eps^2} \| w-1\|_2^2 \lesssim \frac{1}{\eps^2} \|f\|_2^2 \, .
$$
By Lemma \ref{E and 1 over root w} and again \eqref{plancherel with w - 1},
$$
|F_\eps^T\cap R_\eps^T| \leq |\{ |1/\sqrt{w}-1| > \eps \}| \leq |\{ |w-1| > \eps/2 \}| \lesssim \frac{1}{\eps^2} \| w-1\|_2^2\leq \frac{1}{\eps^2} \|f\|_2^2 \, .
$$
It remains to estimate the measure of the set
$F := F_\eps^T \setminus ( (S_\eps^c)\cup R_\eps^T) = (F_\eps^T\cap S_\eps \cap (R_\eps^T)^c$. Let
$$F^{(n)} := \{ s \in F \, : \, T(s)= 2^n \} \, .$$
$F^{(n)}$, $n\in \ZZ$, is a partition of $F$. We will  estimate each $F^{(n)}$ separately. Fix some $n\in \ZZ$ and assume $|F^{(n)}| > 0$. Choose points $s_j\in F^{(n)}$, $j=1,\dots, N$, such that the intervals $I_j:=(s_j-\pi/2^n,s_j+\pi/2^n)$ cover $F^{(n)}$ and no three of them intersect, that is $\sum_{j=1}^N \Z1_{I_j} \leq 2$.

By \eqref{plancherel with w - 1}, we have
$$
\| f\|_{L^2(2^n/3, 2^n)}^2 \gtrsim \| 1/\sqrt{|E_{2^n/3\to 2^n}|} - 1| \|_{L^2(\ZR)}^2 \gtrsim \sum_{j=1}^N \| |E_{2^n/3\to 2^n}| -1 \|_{L^2(I_j)}^2 \, ,
$$
by Lemma \ref{E bigger than eps}, we continue
$$
\gtrsim \sum_{j=1}^N \eps^2 |I_j| \gtrsim \eps^2 |F^{(n)}| \, .
$$
Summing over all $n$, we obtain
$$
\| f\|_{L^2(0,\infty)}^2 \geq 2 \sum_{n\in \ZZ} \| f\|_{L^2(2^n/3, 2^n)}^2 \gtrsim \eps^2 |F^{(n)}| \geq \eps^2  |F| \, .$$
And the proof of the theorem is complete.

\section{Appendix}\label{appendix}
\begin{proof}[Theorem \ref{maintheorem} implies \eqref{lacunary linear Carleson}]
    Fix $f\in L^2$ with $\supp f \in (0,T)$, for some $T>0$. Then, also $f\in L^1$, and for small enough $\varepsilon>0$ we have $\|\varepsilon f \|_1 \leq 1/100$.
    
    Pick an arbitrary $\lambda>0$. By Theorem \ref{maintheorem}, we have, for $\varepsilon < \varepsilon_0 (f,\lambda)$,
    $$
    |\{s\, : \, \sup_n ||E_{\varepsilon f}(2^n,s)|-1| > \varepsilon \lambda \}| \lesssim \frac{1}{\varepsilon^2 \lambda^2 }\|\varepsilon f\|_2^2 = \frac{1}{\lambda^2} \|f\|_2^2 \, .
    $$
    By the linear approximation formulas \eqref{expansionfora}, we have
    $$
    ||E_{\varepsilon f} (2^n,s)|-1| = |\mathcal{F} (\varepsilon f \Z1_{[0,2^n]}) (s)| + O(\varepsilon^2) = \varepsilon |\mathcal{F} ( f \Z1_{[0,2^n]}) (s)| +O(\varepsilon ^2) \, .
    $$
    Plugging this in the above estimate, we have
    $$
    |\{s\, : \, \sup_n |\mathcal{F} ( f \Z1_{[0,2^n]}) (s)| > \lambda + O(\varepsilon) \}| \lesssim \frac{1}{\lambda^2} \|f\|_2^2 \, . 
    $$
    Taking $\varepsilon\to 0$, we conclude
    $$
    |\{s\, : \, \sup_n |\mathcal{F} ( f \Z1_{[0,2^n]}) (s)| > \lambda \}| \lesssim \frac{1}{\lambda^2} \|f\|_2^2 \, . 
    $$
    Taking $T\to +\infty$ and by a triangle inequality and reflection symmetry to allow any $\supp f \subset \ZR$, we conclude \eqref{lacunary linear Carleson} for any $f\in L^2(\ZR)$.
\end{proof}

\begin{proof}[Proof of Lemma \ref{gronwalllemma}]
We start with \eqref{generalizedgronwall}.
Let $E(t, z) = g_1(t) e^{i \phi_1}$ and $E^{\#}(t,z) = g_2 e^{i\phi_2}$. Then, considering \eqref{ODEforE}, we can write
\begin{equation*}
    g_1' + i g_1 \phi '= -i z g_1 + f g_2 e^{i (\phi_2 - \phi_1)}.
\end{equation*}
Taking the real part of the above equation, we get
\begin{equation*}
    g_1' = y g_1 + g_2 (\Re f) \cos (\phi_2 - \phi_1) - g_2(\Im f) \sin (\phi_2 - \phi_1).
\end{equation*}
$g_1$ is away from $0$, so this is equivalent to
\begin{equation*}
    \big( \log g_1 \big)' = y + \frac{g_2}{g_1} (\Re f) \cos (\phi_2 - \phi_1) - \frac{g_2}{g_1}(\Im f) \sin (\phi_2 - \phi_1).
\end{equation*}
As $E$ is Hermite-Biehler, $g_2\leq g_1$ and we get the desired estimate.

To get \eqref{gronwallforscatteringfunction}, we integrate \eqref{51ofpoltoratskiintro}.
\begin{equation*}
    |\mathcal{E} (t_1,z) - \mathcal{E} (t_2,z)| \leq \int_{t_1}^{t_2} e^{-t\Im z} |f| |E^\# (t,z)| dt 
\end{equation*}
\begin{equation*}
    \leq |E(t_1, \bar z)| \int_{t_1}^{t_2} |f| e^{(t-t_1)|\Im z|-t\Im z)} e^{\int_{t_1}^{t_2} |f|} dt \leq e^{(2t_2-t_1)|\Im z| + \int_{t_1}^{t_2}|f|} \int_{t_1}^{t_2} |f|.
\end{equation*}
\end{proof}

\bibliographystyle{amsalpha}

\bibliography{references}

\end{document}